\documentclass[amsthm]{elsart}
\usepackage{amscd}
\usepackage{amsmath}
\usepackage{stmaryrd}
\usepackage{amssymb}
\usepackage{units}
\usepackage[all]{xy}

\usepackage{algorithm}
\usepackage{algorithmic}

\usepackage{savesym}
\savesymbol{AND}
\def\NZQ{\Bbb}               % the font for N,Z,Q,R,C

\def\ZZ{{\NZQ Z}}

\def\frk{\frak}               % font for "Fraktur"

\def\pp{{\frk p}}

\def\mm{{\frk m}}

\def\Phi{{\frk n}}
\def\Phi{{\frk N}}
\def\opn#1#2{\def#1{\operatorname{#2}}} % to make operators
\opn\chara{char} \opn\length{\ell} \opn\pd{pd} \opn\rk{rk}
\opn\projdim{proj\,dim} \opn\injdim{inj\,dim} \opn\rank{rank}
\opn\depth{depth} \opn\sdepth{sdepth} \opn\fdepth{fdepth}
\opn\grade{grade} \opn\height{height} \opn\embdim{emb\,dim}
\opn\codim{codim}  \opn\min{min} \opn\max{max}

\opn\Tr{Tr} \opn\bigrank{big\,rank}
\opn\superheight{superheight}\opn\lcm{lcm}
\opn\trdeg{tr\,deg}%\emph{
\opn\reg{reg} \opn\lreg{lreg} \opn\ini{in} \opn\lpd{lpd}
\opn\size{size}
\opn\div{div} \opn\Div{Div} \opn\cl{cl} \opn\Cl{Cl}
\opn\Spec{Spec} \opn\Supp{Supp} \opn\supp{supp} \opn\Sing{Sing}
\opn\Ass{Ass} \opn\Min{Min}
\opn\Ann{Ann} \opn\Rad{Rad} \opn\Soc{Soc}
\opn\Im{Im} \opn\Ker{Ker} \opn\Coker{Coker} \opn\Am{Am}
\opn\Hom{Hom} \opn\Tor{Tor} \opn\Ext{Ext} \opn\End{End}
\opn\Aut{Aut} \opn\id{id}  \opn\deg{deg}

\opn\nat{nat}
\opn\pff{pf}%   \pf exists already
\opn\Pf{Pf} \opn\GL{GL} \opn\SL{SL} \opn\mod{mod} \opn\ord{ord}
\opn\Gin{Gin} \opn\Hilb{Hilb}
\opn\aff{aff} \opn\con{conv} \opn\relint{relint} \opn\st{st}
\opn\lk{lk} \opn\cn{cn} \opn\core{core} \opn\vol{vol}
\opn\gr{gr}

\def\pot#1#2{#1[\kern-0.28ex[#2]\kern-0.28ex]}

\newcommand{\fracs}[2]{\displaystyle\frac{#1}{#2}}

\opn\dirlim{\underrightarrow{\lim}}
\opn\inivlim{\underleftarrow{\lim}}
\let\to=\rightarrow

\def\Implies{\ifmmode\Longrightarrow \else
        \unskip${}\Longrightarrow{}$\ignorespaces\fi}
\def\implies{\ifmmode\Rightarrow \else
        \unskip${}\Rightarrow{}$\ignorespaces\fi}
\def\iff{\ifmmode\Longleftrightarrow \else
        \unskip${}\Longleftrightarrow{}$\ignorespaces\fi}

\let\:=\colon
\let\epsilon\varepsilon
\let\phi=\varphi
\let\kappa=\varkappa
\def\qed{\ifhmode\textqed\fi
      \ifmmode\ifinner\quad\qedsymbol\else\dispqed\fi\fi}
\def\textqed{\unskip\nobreak\penalty50
       \hskip2em\hbox{}\nobreak\hfil\qedsymbol
       \parfillskip=0pt \finalhyphendemerits=0}
\def\dispqed{\rlap{\qquad\qedsymbol}}

\opn\dis{dis}
\def\pnt{{\raise0.5mm\hbox{\large\bf.}}}

\opn\Lex{Lex}

\begin{document}

\begin{frontmatter}

\title{\bf Constructive General Neron Desingularization for one dimensional local rings}

\thanks{This paper was partially written during the visit of the first author at the Institute of Mathematics Simion Stoilow of the Romanian Academy supported by a BitDefender Invited Professor Scholarship, 2015. The  support from the Department of Mathematics of the University of Kaiserslautern of the first author and the support from the project  ID-PCE-2011-1023, granted by the Romanian National Authority for Scientific Research, CNCS - UEFISCDI  of the second author are gratefully acknowledged. }

\author{Gerhard Pfister}
\address{Gerhard Pfister,  Department of Mathematics, University of Kaiserslautern, Erwin-Schr\"odinger-Str., 67663 Kaiserslautern, Germany}
\ead{pfister@mathematik.uni-kl.de}

\author{Dorin Popescu}
\address{Dorin Popescu, Simion Stoilow Institute of Mathematics of the Romanian Academy, Research unit 5,
University of Bucharest, P.O.Box 1-764, Bucharest 014700, Romania}
\ead{dorin.popescu@imar.ro}

%\maketitle

\begin{abstract} An algorithmic proof of General Neron Desingularization is given here for one dimensional local rings and it is implemented in \textsc{Singular}. Also a theorem recalling Greenberg' strong approximation theorem is presented for one dimensional  local rings.
\end{abstract}

\begin{keyword}
Smooth morphisms,  regular morphisms, smoothing ring morphisms.\\
 {\it 2010 Mathematics Subject Classification: Primary 13B40, Secondary 14B25,13H05,13J15.}
\end{keyword}

\end{frontmatter}

\section{Introduction}

A ring morphism $u:A\to A'$ has  {\em regular fibers} if for all prime ideals $P\in \Spec A$ the ring $A'/PA'$ is a regular  ring, i.e. its localizations are regular local rings. It has {\em geometrically regular fibers}  if for all prime ideals $P\in \Spec A$ and all finite field extensions $K$ of the fraction field of $A/P$ the ring  $K\otimes_{A/P} A'/PA'$ is regular.
If for all $P\in \Spec A$ the fraction field of $A/P$ has characteristic $0$ then the regular fibers of $u$ are geometrically regular fibers. A flat morphism $u$ is {\em regular} if its fibers are geometrically regular.  If $u$ is regular of finite type then $u$ is called {\em smooth}.  A localization of a smooth algebra is called {\em essentially smooth}.

In Artin approximation theory (introduced in \cite{A}) an important result (see \cite{P4}) is the following theorem, generalizing the Neron Desingularization \cite{N}, \cite{A}.
\vskip 0.3 cm
\begin{thm}[General Neron Desingularization, Popescu \cite{P0}, \cite{P}, \cite{P1},  Andr\'e \cite{An}, Swan \cite{S}, Spivakovsky \cite{Sp}\label{gnd}]  Let $u:A\to A'$ be a  regular morphism of Noetherian rings and $B$ a finite type $A$-algebra. Then  any $A$-morphism $v:B\to A'$   factors through a smooth $A$-algebra $C$, that is $v$ is a composite $A$-morphism $B\to C\to A'$.
\end{thm}
\vskip 0.3 cm
The purpose of this paper is to give an algorithmic proof of the above theorem when $A,A'$ are one dimensional local rings, that is Theorem \ref{m}. When $A,A'$ are domains such algorithm is given in \cite{AP}, the case when $A,A'$ are discrete valuation rings proved by N\'eron \cite{N} is given in a different way in \cite{P3} with applications in arcs frame. The present algorithm was implemented by the authors in the Computer Algebra system \textsc{Singular} \cite{Sing} and will be as soon as possible  found in a development version at \begin{verbatim}https://github.com/Singular/Sources/blob/spielwiese/Singular/LIB/.\end{verbatim}

The proof of Theorem \ref{m} splits essentially in three steps. We will give here the idea in case $A$ and  $A'$ are domains. In step 1 we reduce the problem to the case when $H_{B/A}\cap A\not =0$, $H_{B/A}$ being the ideal defining the nonsmooth locus of $B$ over $A$. Let $0\not =d\in H_{B/A}\cap A$. This means geometrically that $\Spec B_d\to \Spec A_d$ is smooth. In the second step we construct a smooth $A$-algebra $D$, $A\subset D\subset A'$ and an $A$-morphism $v':B\to D/d^3D$ such that $v\equiv v'$ modulo $d^3A'$. If $A'$ is the completion $\hat A$ of $A$ we can use $D=A$. The third step resolves the singularity.
If $B=A[Y]/I$, $Y=(Y_1,\ldots,Y_n)$ then we can find  $f=(f_1,\ldots,f_r)$, $r\leq n$  a system of polynomials from $I$ (given a tuple $(b_1,\ldots b_s) $ we denote usually by the corresponding unindexed letter $b$ the vector  $(b_1,\ldots b_s) $), and an  $r\times r$-minor $M$  of the Jacobian matrix $(\partial f_i/\partial Y_j)$ such that $d\equiv MN$ modulo $I$ for some $N\in ((f):I)$, where $(f)$ denotes the ideal generated by the system $f$. Then $v'(MN)=ds$ for some $s\in 1+dD$.  Assume that $M=\det(\partial f_i/\partial Y_j)_{1\leq i,j\leq r}$. Let $H$ be the matrix obtained by adding to  $(\partial f/\partial Y)$ the boarder block $(0|\mbox{Id}_{n-r})$ and let $G'$ be the adjoined matrix of $H$ and $G=NG'$. Consider in $D[Y,T]$, $T=(T_1,\ldots, T_n)$, the ideal $J=((f,s(Y-y')-dG(y')T):d^2)$, where $y'\in D^n$ is lifting $v'(Y)$.
Then $C$ is a suitable localization of the $B\otimes_A D$-algebra $D[Y,T]/(I,J)$ and $v$ extends to $C$ by $v(T)=t=(1/d^2)H(y')(v(Y)-y')$.

Consider the following example.  Let $A={\bf Q}[x]_{(x)}$, $A'={\bf C}[[x]]$, $B=A[Y_1,Y_2]/(Y_1^2+Y_2^2)$, $a\in {\bf C}$  a transcendental element over $\bf Q$, ${\bar u}\in {\bf C}[[x]]\setminus {\bf C}[x]_{(x)}$ and $u=a+x^6{\bar u}$. Let $v$ be given by $v(Y_1)=xu$,  $v(Y_2)=xiu$, where $i=\sqrt{-1}$. In step 1 we change $B$ by $B_1=A[Y_1,Y_2,Y_3]/I$, $I=(Y_1^2+Y_2^2, x-2Y_1Y_3)$ and extend $v$ by $v(Y_3)=1/(2u)$. We have $4Y_1^2Y_3^2\in H_{B/A}$ which implies $d=x^2\in H_{B/A}\cap A$. We define $D=A[a,a^{-1},i]$ and $v'(Y)=y'=(xa,xia,1/(2a))$.
This is step two.

To understand step 3 we simplify the example taking $B=A[Y_1,Y_2]/(Y_1Y_2-x^2)$, $u=1+x^6{\bar u}$ and $v$ given by $v(Y_1)=xu$, $v(Y_2)=x/u$. Then $d=x\in H_{B/A}$,
$D=A$, $y'_1=x=y'_2$. We obtain  $H=\begin{pmatrix}Y_2 & Y_1\\ 0&1\end{pmatrix}$, $G=G'=\begin{pmatrix}1 & -Y_1\\ 0&Y_2\end{pmatrix}$,
$N=1$ and $J=((Y_1Y_2-x^2, Y_1-x-xT_1+x^2T_2, Y_2-x-x^2T_2):x^2)$. We have $J=(xT_1T_2-x^2T_2^2+T_1, Y_1-x-xT_1+x^2T_2, Y_2-x-x^2T_2)$ and we obtain that $C\cong (A[T_1,T_2]/(xT_1T_2-x^2T_2^2+T_1))_{1+xT_2}\cong (A[T_2])_{1+xT_2}$ is a smooth $A$-algebra.

When $A'$ is the completion of a Noetherian local ring $A$ of dimension one we show that we may have a linear Artin function as it happens in the Greenberg's case (see \cite{Gr} and \cite[Theorem 18]{AP}). More precisely, the Artin function is given by $c\to (\rho+1)(e+1)+c$, where $e,\rho$ depend on $A$ and the polynomial system of equations defining $B$ (see Theorem \ref{gr}).

We thank to both Referees for useful  comments and pointing some misprints.

\section{Theorem  \ref{gnd} in one dimensional local rings}

 Let $u:A\to A'$ be a flat morphism of Noetherian local rings of dimension $1$. Suppose that  the maximal ideal $\mm$ of $A$ generates the maximal ideal of $A'$. Moreover suppose that $u$ is a regular morphism, $k$ is infinite and  there exist canonical inclusions $k=A/\mm\subset A$, $k'=A'/\mm A'\subset A'$ such that $u(k)\subset k'$.

  Let $B=A[Y]/I$, $Y=(Y_1,\ldots,Y_n)$. If $f=(f_1,\ldots,f_r)$, $r\leq n$ is a system of polynomials from $I$ then we can define the ideal $\Delta_f$ generated by all $r\times r$-minors of the Jacobian matrix $(\partial f_i/\partial Y_j)$.   After Elkik \cite{El} let $H_{B/A}$ be the radical of the ideal $\sum_f ((f):I)\Delta_fB$, where the sum is taken over all systems of polynomials $f$ from $I$ with $r\leq n$.
Then $B_P$, $P\in \Spec B$ is essentially smooth over $A$ if and only if $P\not \supset H_{B/A}$ by the Jacobian criterion for smoothness.
   Thus  $H_{B/A}$ measures the non smooth locus of $B$ over $A$.
  $B$ is {\em standard smooth} over $A$ if  there exists  $f$ in $I$ as above such that $B= ((f):I)\Delta_fB$.

  The aim of this paper is to give an easy algorithmic  proof of the following theorem.
\vskip 0.3 cm
 \begin{thm} \label{m} Any $A$-morphism $v:B\to A'$   factors through a standard smooth $A$-algebra $B'$.
 \end{thm}
\vskip 0.3 cm
We consider  in the algorithmic part the following assumption, which we will keep in the whole paper

  $(*)$ $A$ is essentially of finite type over a field $k$, let us say $A\cong (k[x]/J)_{(x)}$ for some variables $x=(x_1,\ldots x_m)$, and the completion of $A'$ is $k'[[x]]/(J)$.

  When $v$ is defined by polynomials $y$ from $k'[x]$ then our problem is easy. Let $L$ be the field obtained by adjoining to $k$ all coefficients of $y$. Then $R=(L[x]/(J))_{(x)}$ is a subring of $A'$ containing $\Im v$ which is essentially smooth over $A$. Then we may take $B' $ as a standard smooth $A$-algebra such that $R$ is a localization of $B'$.
Thus we will not suppose in this paper that $y$ is polynomial and therefore $L$ is not necessarily a finite type field extension of $k$.

In the proof we need to know that $v(H_{B/A})$ is not contained in any minimal prime ideal of $A'$. In theory, we may reduce to this case as it follows. Let $\pp\in \Min A'$. Since $u$ is regular, it induces a regular map $u_{\pp}:A_{(\pp\cap A)}\to A'_{\pp}$ of local Artinian rings
(in particular $k(\pp)\otimes_{A_{\pp\cap A}}u_p$, $k(\pp)= A_{(\pp\cap A)}/(\pp\cap A)A_{(\pp\cap A)} $ is a separable field extension). Note that $A_{(\pp\cap A)}\supset k(\pp)$ because of $(*)$ and $A'_{\pp}$ is a filtered inductive limit of its subrings of the form $E_{F_{\pp}}=A_{(\pp\cap A)}\otimes_{k(\pp)}F_{\pp}$ for all finite type field subextension  $F_{\pp}/k(\pp)$ of $(A'_{\pp}/\pp A'_{\pp})/k(\pp)$. We may change $B$ by a finite type $B$-algebra $\tilde B$ of $A'$ such that ${\tilde B}_{\pp\cap {\tilde B}}\cong E_{F_{\pp}}$. It follows that $v(H_{B/A})\not \subset \pp A'$ for all $\pp\in \Min A'$. Next we will assume from the beginning that

$(**)$ $v(H_{B/A})\not \subset \pp A'$ for all $\pp\in \Min A'$.

Unfortunately, the computer cannot decide this since we are not able to give the whole information concerning the coefficients of $y$.   But we are able to decide if $v(H_{B/A})$ is not contained  in $\mm^NA'+\pp$ for $N>>0$ and one $\pp\in \Min A'$. In the following we suppose that  $v(H_{B/A})\not \subset \mm^NA'+\pp$ for all $\pp\in \Min A'$ and a certain $N>>0$. Choose $\gamma\in H_{B/A}$ such that $v(\gamma)$ is not in $\cup_{\pp\in \Min A'} \pp +\mm^NA'$.

The idea of the proof of Theorem \ref{m} is  to find $f=(f_1,\ldots,f_r)$ in $I$ and a $d\in v(((f):I)\Delta_f)A'\cap A$ which is not in  $\cup_{\pp\in \Min A'} \pp $. The assumption $(**)$ gives just that there exists $f_{\pp}$ for any $\pp\in \Min A'$ such that $\Delta_{f_{\pp}}((f_{\pp}):I)\not \subset   \pp A'$ for all $\pp\in \Min A'$. The main problem is to reduce to the case when $f_{\pp}$ does not depend of $\pp$. Actually, this follows if $I_{\gamma}/I^2_{\gamma}$ is free over $B_{\gamma}$.  In the next three lemmas containing some results of Elkik \cite{El} (in the form used in \cite{P0}, \cite{S}, \cite{P2}), we see that this is true if we reduce to the case when $\Omega_{B_{\gamma }}/A$ is free over $B_{\gamma}$. In general, the last module is projective but not free as the following example shows.
\vskip 0.3 cm
\begin{exmp} \label{k}
Let $k$ be a subfield of $\bf R$, $A=(k[x_1,x_2,x_3]/(x_1^2-x_2x_3,x_3^2-x_1x_2))_{(x_1,x_2,x_3)}$, and $\alpha=x_1Y_1^2+x_2Y_2^2+x_3Y_3^2-x_1-x_2-x_3\in A[Y]$, $Y=(Y_1,Y_2,Y_3)$. Set $f_1=x_2\alpha=x_3^2Y_1^2+x_2^2Y_2^2+x_1^2Y_3^2-x_3^2-x_2^2-x_1^2$, $f_2=x_1\alpha$, $f_3=x_3\alpha$ and $I=(f)$. Then $x_2I\subset (f_1)$, $x_1I\subset (f_2)$, $x_3I\subset (f_3)$.
Also note that $x_2^2Y_2\in \Delta_{f_1}$, $x_1^2Y_1\in \Delta_{f_2}$, $x_3^2Y_3\in \Delta_{f_3}$.

Let $B=A[Y]/I$, $A'=k'[[x_1,x_2,x_3]]/(x_1^2-x_2x_3,x_3^2-x_1x_2)$, where $k\subset k'$ is a field extension. Let $u_1,u_2$ be two algebraically independent elements of $k'[[x_1,x_2,x_3]]$ over $k[x_1,x_2,x_3]$. Set $y_1=x_3u_1-1$, $y_2=x_3u_2-1$ and we may find $y_3$ such that $y_3^2=1-x_1x_3u_1^2+2x_1u_1-x_2x_3u_2^2+2x_2u_2$. Clearly, $\alpha(y)=0$, that is $\alpha(y_i)=0$ for all $1\leq i\leq n$, and so we get an
$A$-morphism $v:B\to A'$ given by $Y\to y$.

Note that $H_{B/A}=(x_1,x_2,x_3)$.
Take $\gamma=x_1+x_2+x_3$ and $\beta=f_1+f_2+f_3$. Then $I_{\gamma}=(\beta)_{\gamma}=(\alpha)_{\gamma}$ and  we claim that $\Omega_{B_{\gamma}/A}$
is projective but not free.
Indeed, $\Omega_{B/A} =BdY_1\oplus BdY_2\oplus BdY_3/(x_1Y_1dY_1+x_2Y_2dY_2+x_3Y_3dY_3) $ and $\Omega_{B_{\gamma}/A}$ is projective because its Fitting  ideal  is $(x_1Y_1,x_2Y_2,x_3Y_3)B_{\gamma}\supset (x_1Y_1^2+x_2Y_2^2+x_3Y_3^2)B_{\gamma}  =\gamma B_{\gamma}=B_{\gamma}$ (see e.g. \cite[Proposition 1.3.8]{JP}).

Now suppose that    $\Omega_{B_{\gamma}/A}$ is free over $B_{\gamma}$. Then $\lambda=x_1Y_1+x_2Y_2+x_3Y_3$ can be included in a basis of $B_{\gamma}dY_1\oplus B_{\gamma}dY_2\oplus B_{\gamma}dY_3$.  More precisely, there exists a $3\times 3$ invertible matrix $(a_{ij})$, $a_{ij}=a_{ij}(x,Y)$ over $B_{\gamma}$ with $a_{1j}=x_jY_j$ for $j\in [3]$, let us say $a_{ij}=b_{ij}(x,Y)/c(x)$,  $j=2,3$ with $b_{ij}\in k[x,Y]$, $c\in k[x]$. We may choose some positive real numbers $x'_1.x'_2,x'_3$ such that
$x'^2_1=x'_2x'_3$, $x'^2_3=x'_1x'_2$, $c(x')\not =0$ and the matrix $(a_{ij}(x',Y))$ invertible in $B'={\bf R}[Y]/(x'_1Y_1^2+x'_2Y_2^2+x'_3Y_3^2-x'_1-x'_2-x'_3)$. It follows that $P'=(B'dY_1\oplus B'dY_2\oplus B' dY_3)/<x'_1 Y_1dY_1+x'_2 Y_2dY_2+x'_3Y_3 dY_3>$
is free over $B'$. Changing $Y_i$ by $\sqrt{x'_i/(x'_1+x'_2+x'_3)}Y_i$ we see that over $B''={\bf R}[Y]/(Y_1^2+Y_2^2+Y_3^2-1)$ the module $P''=(B''dY_1\oplus B''dY_2\oplus B'' dY_3)/<Y_1 dY_1+Y_2 dY_2+Y_3 dY_3>$ is free, that is the tangent bundle over the real sphere is trivial which contradicts for example \cite[page 114]{K}.

On the other hand, note that
$\Omega_{B_{\gamma}[Y_4]/A} =$
$$B_{\gamma}[Y_4]dY_1\oplus B_{\gamma}[Y_4]dY_2\oplus B_{\gamma}[Y_4]dY_3\oplus B_{\gamma}[Y_4]dY_4 /<x_1Y_1dY_1+x_2Y_2dY_2+x_3Y_3dY_3> $$
is free because in $\sum_{i=1}^4 B_{\gamma}[Y_4]dY_i$ we have the basis $$\{x_1Y_1dY_1+x_2Y_2dY_2+x_3Y_3dY_3, dY_1-Y_1dY_4,  dY_2-Y_2dY_4,  dY_3-Y_3dY_4\}.$$

Let $B_1$ be the symmetric algebra $S_B(I/I^2)$ of $I/I^2$ over $B$. Then $(B_1)_{\gamma}$ is the symmetric algebra $S_{B_{\gamma}}((I/I^2)_{\gamma})$ of $(I/I^2)_{\gamma}$ over $B_{\gamma}$. But $(I/I^2)_{\gamma}$ is free generated by $\alpha$
and so $(B_1)_{\gamma}\cong B_{\gamma}[Y_4]$. Consequently, $\Omega_{(B_1)_{\gamma}/A}$ is free from above.
\end{exmp}

 \begin{lem}(\cite[Lemma 3.4]{P0}) \label{e1} Let $B_1$ be the symmetric algebra $S_B(I/I^2)$ of $I/I^2$ over\footnote{Let $M$ b e a finitely represented $B$-module and $B^m\xrightarrow{(a_{ij})} B^n\to M\to 0$ a presentation then $S_B(M)=B[T_1, \ldots, T_n]/J$ with $J=(\{\sum\limits^n_{i=1} a_{ij} T_i\}_ {j=1, \ldots, m})$.} $B$. Then $H_{B/A}B_1\subset H_{B_1/A}$ and  $(\Omega_{B_1/A})_{\gamma}$ is free over $(B_1)_{\gamma}$.
\end{lem}

\begin{lem} (\cite[Proposition 4.6]{S}) \label{e2}  Suppose that  $(\Omega_{B/A})_{\gamma}$ is free over $B_{\gamma}$. Let $I'=(I,Y')\subset A[Y,Y']$, $Y'=(Y'_1,\ldots,Y'_n)$. Then $(I'/I'^2)_{\gamma}$ is free over $B_{\gamma}$.
\end{lem}

\begin{lem} (\cite[Corollary 5.10]{P2}) \label{e3} Suppose that $(I/I^2)_{\gamma}$ is free over  $B_{\gamma}$. Then a power of $\gamma$ is in  $ ((g):I)\Delta _g$ for some $g=(g_1,\ldots g_r)$, $r\leq n$ in $I$.
\end{lem}

{\bf Step 1.}  Reduction to the case when $\Omega_{B_{\gamma}/A}$ is free over $ B_{\gamma}$ .

Let $B_1$ be given by Lemma \ref{e1}. The inclusion $B\subset B_1$ has a retraction $w$ which maps $I/I^2$ to zero. For the reduction we change $B,v$ by $B_1,vw$.

{\bf Step 2.} Reduction to the case when  $(I/I^2)_{\gamma}$ is free over $ B_{\gamma}$.

Since $\Omega_{B_{\gamma}/A}$ is free over $ B_{\gamma}$ we see using Lemma \ref{e2} that changing $I$ with $(I,Y')\subset A[Y,Y']$ we may suppose that $(I/I^2)_{\gamma}$ is free over $ B_{\gamma}$.

{\bf Step 3.} Reduction to the case when a power of $\gamma$ is in  $ ((f):I)\Delta _f$ for some $f=(f_1,\ldots f_r)$, $r\leq n$ in $I$.

We reduced to the case when $(I/I^2)_{\gamma}$ is free over $ B_{\gamma}$. Then it is enough to  use Lemma \ref{e3}.

{\bf Step 4.} The Jacobian matrix $(\partial f/\partial Y)$ can be completed with $(n-r)$ rows from $k^n$ obtaining a square $n$ matrix $H$ with $v(\det H) \not \in \cup_{\pp\in \Min A'} \pp$.

We may suppose that $r<n$, otherwise there exist nothing to show. Note that the rows of  $(\partial f/\partial Y)$ are mapped by $v$ in $r$ linear independent vectors from $(A'/\pp)^n$ for each $\pp\in \Min A'$. Fix a $\pp$ and consider the set $\Lambda_{\pp}$ of all $(n-r)$ linear independent vectors from $k^n$ which define a basis in $Q(A'/\pp)^n$ together with   the rows of  $v(\partial f/\partial Y)$. Clearly $L_{\pp}$ is a nonempty open Zariski set of $k^{n(n-r)}$. Since $k^{n(n-r)}$ is irreducible we get $\cap_{\pp\in \Min A'} L_{\pp}\not = \emptyset$. Choosing
$(n-r)$ rows from  $\cap_{\pp\in \Min A'} L_{\pp}$ we may complete $(\partial f/\partial Y)$ to the wanted matrix $H$.

{\bf Step 5.} Reduction to the case when $((\det H)((f):I))\cap A$ is $\mm$ primary.

By Step 3 there exists a polynomial $R'\in ((f):I)$ such that $v(R')$ is not in $\cup_{p\in \Min A'} p$. Set $P'=R'\det H$. Then $v(P')$ generates in $A'$ an ideal of height $1$ which must be   $\mm A'$ primary. Then $(v(P'))\cap A$ is $\mm$ primary too and we may choose $d'\in (v(P'))\cap A$ such that $d'A$ is $\mm$ primary, let us say $d'=v(P')z$ for some $z\in A'$.
 Set $B_1=B[Z]/(f_{r+1}) $, where $f_{r+1}=-d'+P'Z$
and let $v_1:B_1\to A'$ be the map of $B$-algebras given by $Z\to z$. It follows that $d'\in (( f,f_{r+1}):(I,f_{r+1}))$ and  $d'\in \Delta_f$, $d'\in \Delta_{f_{r+1}}$. Then $d=d'^2\equiv P\ \mbox{modulo}\ (I,f_{r+1})$ for $P=P'^2Z^2\in H_{B_1/A}$. For the reduction change $B$ by $B_1$ and $H$ by  $
\left(\begin{array}{cc}
H& 0 \\
* & P'
\end{array}\right).$ Note that $d\in ((\det H)((f):I))\cap A$. The determinant of the  new $H$ is  the determinant of the old $H$ multiplied with $P'$. Thus $P$ is the determinant of the new $H$ multiplied with $R=R'Z^2$

\begin{rem} In Example \ref{k} the module $\Omega_{B_{\gamma}/A }$ is not free and so we may apply Step 1. In fact we do not need to apply the Steps 1, 2, 3 because $(I/I^2)_{\gamma}$ is already free.
\end{rem}

\section{Proof of the case  when $((\det H)((f):I))\cap A$ is $\mm$ primary.}

Let $(0)=\cap_{\pp\in \Ass A} Q_{\pp}$ be a reduced primary decomposition \footnote{The primary decomposition of an ideal in a polynomial ring and its localizations by a maximal ideal can be computed using the SINGULAR library primdec.lib. This is a very difficult task in the computational algebra usually needing a lot of Gr\"obner basis computations with respect to the lexicographical ordering.}
of $(0)$ in $A$, where $Q_{\pp}$ is a primary ideal with $\sqrt{Q_{\pp}}=\pp$. Let $d$ be defined as in the end of Section 2. Define $e$ by $(0:_Ad^e)=(0:_Ad^{e+1})$. This equality happens  for example  taking $e$ such that ${\pp}^e\subset Q_{\pp}$ for all $\pp\in \Ass A$.
 Set   $\bar A=A/(d^{2e+1})$, $\bar A'=A'/d^{2e+1}A'$, $\bar u=\bar A\otimes_Au$, $\bar B=B/d^{2e+1}B$, $\bar v=\bar A\otimes_Av$. By base change $\bar u$ is a regular morphism of Artinian  local rings.

{\bf Step 6.} There exists a smooth $A$-algebra and an $A$-morphism $\omega:D\to A'$ such that $y=v(Y)\in \Im \omega +d^{2e+1}A'$.

We extend the proof of \cite[Theorem 10]{P3} in our case. But now $A'$ is not the completion of $A$, that is the coefficients of $y$ in $x$ are not necessarily from $k$. Fortunately, as in the proof of  \cite[Theorem 10]{P3} we need only a finite number of this coefficients, namely those of monomials $x$ which are not in $d^{2e+1}A'$. This is the reason to ask for the existence of such $D,\omega$.

 By  \cite[19,7.1.5]{G}  for every field extension $L/k$ there exists a flat complete Noetherian local $\bar A$-algebra $\tilde A$, unique  up to an isomorphism, such that $\mm\tilde A$ is the maximal ideal of $\tilde A$ and $\tilde A/\mm\tilde A\cong L$. It follows that $\tilde A$ is Artinian. On the other hand, we may consider the localization $A_L$ of $L\otimes_k\bar A$ in $\mm(L\otimes_k\bar A)$ which is Artinian and so complete. By uniqueness we see that $A_L\cong \tilde A$.  It  follows that $\bar A'\cong A_{k'}$. Note that $A_L$ is essentially smooth over $A$ by base change and $\bar A'$ is a filtered union of sub-$\bar A$-algebras $A_L$ with $L/k$ finite type field sub extensions of $k'/k$.

  Choose $L/k$ a finite type field extension such that $A_L$ contains the residue class ${\bar y}\in \bar A'^n$ induced by $y$.
 In fact ${\bar y}$ is a vector of polynomials in the generators of $\mm$ with the coefficients $c_{\nu}$ in $k'$ and we may take $L=k((c_{\nu})_{\nu})$. Then $\bar v$ factors through $A_L$. Assume that $k[(c_{\nu})_{\nu}]\cong k[(U_{\nu})_{\nu}]/\bar J$ for some new variables $U$ and a prime ideal  $\bar J\subset k[U]$. We have $H_{L/k}\not= 0$ because $L/k$ is separable. Then we may assume that there exist $w=( w_1,\ldots,w_p)$ in ${\bar J}^p$ such that
 $\rho=\det(\partial w_i/\partial U_{\nu})_{i,\nu\in [p]}\not =0$ and a nonzero polynomial $ \tau\in ((w):\bar J)\setminus \bar J$ (we set $[p]=\{1,\ldots,p\}$). Actually, we may reduce to the case when $p=1$, but this means a complication for our algorithm. Thus $L$ is a fraction ring of the smooth $k$-algebra $(k[U]/( w))_{ \rho \tau}$. Note that $w, \rho, \tau$ can be considered in $ A[U]$ because $k\subset  A$ and $c_{\nu}\in A'$ because $k'\subset A'$.

   Then $\bar v$ factors through a smooth $\bar A$-algebra $C\cong (\bar A[U]/(w))_{\rho\tau\gamma}$ for some polynomial $\gamma$ which is not in $\mm (\bar A[U]/(w))_{\rho\tau}$.

\begin{lem}\label{D} There exists a smooth $A$-algebra $D$ such that $\bar v$ factors through $\bar D=\bar A\otimes_A D$.
\end{lem}
\begin{proof} By our assumptions $u(k)\subset k'$. Set $D=( A[U]/(w))_{\rho\tau\gamma}$ and $\omega:D\to A'$ be the map given by $U_{\nu}\to c_{\nu}$. We have $C\cong {\bar A}\otimes_AD$.
Certainly, $\bar v$ factors through  $\bar \omega=\bar A\otimes_A\omega $ but in general $v$ does not factor through $\omega$.
\hfill\ \end{proof}

It is worth  recalling the following two remarks from \cite{AP}.

\begin{rem} \label{rem} If $A'=\hat A$ then $\bar A\cong \bar A'$ and we may take $D=A$.
\end{rem}

\begin{rem} If $k\subset A$ but $L\not \subset A'$ then $D=(A[U,Z]/(w-d^{2e+1}Z))_{\rho\tau\gamma}$, $Z=(Z_1,\ldots,Z_p)$, $U=(U_1,\ldots,U_q)$ is a smooth $A$-algebra and $\bar D\cong C[Z]$. Note that   $\bar v$ factors through a map $C\to \bar A'$ given let us say by $U\to \lambda+d^{2e+1}A'^q$ for some $\lambda$ in $A'^q$. Thus  $w(\lambda) =d^{2e+1}z$ for some $z$ in $A'^p$ and we get a map
 $\omega:D\to A'$, $(U,Z)\to (\lambda,z)$.  As above $\bar v$ factors through  $\bar \omega=\bar A\otimes_A\omega $ but in general $v$ does not factor through $\omega$.
If also $k\not \subset A$ then the construction of $D$ goes  using a lifting of $w,\tau, \gamma$ from $k[U]$ to $A[U]$. In both cases we may use $D$ as it follows.
\end{rem}

 Let $\delta:B\otimes_AD\cong D[Y]/ID[Y]\to A'$ be the $A$-morphism given by $b\otimes \lambda\to v(b)\omega(\lambda)$.

{\bf Step 7.} $\delta$ factors through a special finite type $B\otimes_AD$-algebra $E$.

Note that the map $\bar B\to \bar D$ is given by $Y\to y'+d^{2e+1}D$ for some $y'\in D^n$. Thus $I(y')\equiv 0$ modulo $d^{2e+1}D$. Since $\bar v$ factors through $\bar \omega$ we see that $\bar \omega(y'+d^{2e+1}D)=\bar y$. Set $\tilde y=\omega(y')$. We get
  $y-\tilde y=v(Y)-\tilde y\in d^{2e+1}A'^n$, let us say $y-\tilde y=d^{e+1}\epsilon$ for $\epsilon\in d^eA'^n$.

Recall  that $P=R\det H$ for  $R\in ((f):I)$ (see Step 5). We have $d\equiv P$ modulo $I$ and so $P(y')\equiv d$ modulo $d^{2e+1}$ in $D$ because $I(y')\equiv 0$ modulo $d^{2e+1}D$. Thus $P(y')=ds$ for a certain $s\in D$ with $s\equiv 1$ modulo $d$.
 Let $G'$ be the adjoint matrix of $H$ and $G=RG'$. We have
$GH=HG=P\mbox{Id}_n$
and so
$$ds\mbox{Id}_n=P(y')\mbox{Id}_n=G(y')H(y').$$

But $H$ is the matrix $(\partial f_i/\partial Y_j)_{i\in [r],j\in [n]}$ completed with some $(n-r)$ rows from $k^n$. Especially we obtain
 \begin{equation}\label{identity}(\partial f/\partial Y)G=(P\mbox{Id}_r|0).\end{equation}

 Then $t:=H(y')\epsilon\in d^eA'^n$
satisfies
$$G(y')t=P(y')\epsilon=ds\epsilon$$
 and so
 $$s(y-\tilde y)=d^e\omega(G(y'))t.$$
 Let
 \begin{equation}\label{def of h}h=s(Y-y')-d^eG(y')T,\end{equation}
 where  $T=(T_1,\ldots,T_n)$ are new variables. The kernel of the map
$\phi:D[Y,T]\to A'$ given by $Y\to y$, $T\to t$ contains $h$. Since
$$s(Y-y')\equiv d^eG(y')T\ \mbox{modulo}\ h$$
and
$$f(Y)-f(y')\equiv \sum_j\fracs{\partial f}{\partial Y_j}(y') (Y_j-y'_j)$$
modulo higher order terms in $Y_j-y'_j$, by Taylor's formula we see that for $p=\max_i \deg f_i$ we have
\begin{equation}\label{def of Q}s^pf(Y)-s^pf(y')\equiv  \sum_js^{p-1}d^e\fracs{\partial f}{\partial Y_j}(y') G_j(y')T_j+d^{2e}Q\end{equation}
modulo $h$ where $Q\in T^2 D[T]^r$.   We have $f(y')=d^{e+1}b$ for some $b\in d^eD^r$. Then
\begin{equation}\label{def of g}g_i=s^pb_i+s^pT_i+d^{e-1}Q_i, \qquad i\in [r] \end{equation}  is in the kernel of $\phi$. Indeed,  we have $s^pf_i=d^{e+1}g_i\ \mbox{modulo}\ h$ because of (\ref{identity}) and $P(y')=ds$. Thus
$d^{e+1}\phi(g)=d^{e+1}g(t)\in (h(y,t),f(y))=(0)$. Since $Q\in T^2D[T]^r$ and $t\in d^eA'^n$ we get $g(t)\in d^eA'^r$ and so $g(t)\in (0:_{A'}d^{e+1})\cap d^eA'=0$, because $u$ is flat and $(0:_{A'}d^e)=(0:_{A'}d^{e+1})$. Set $E=D[Y,T]/(I,g,h)$ and let  $\psi:E\to A'$ be the map induced by $\phi$. Clearly, $v$ factors through $\psi$ because $v$ is the composed map $B\to B\otimes_AD\cong D[Y]/I\to E\xrightarrow{\psi} A'$.

{\bf Step 8.} There exist $s',s''\in E$ such that $E_{ss's''}$ is smooth over $A$ and $\psi$ factors through $E_{ss's''}$.

Note that the $r\times r$-minor  $s'$ of $(\partial g/\partial T)$ given by the first  $r$-variables $T$ is from $s^{rp}+(T)\subset 1+(d,T)$ because $Q\in (T)^2$. Then $V=(D[Y,T]/(h,g))_{ss'}$ is smooth over $D$. We claim that $I\subset (h,g)D[Y,T]_{ss's''}$ for some other $s''\in 1+(d,T)D[Y,T]$. Indeed, we have $PI\subset (h,g)D[Y,T]_s$ and so $P(y'+s^{-1}d^eG(y')T)I\subset (h,g)D[Y,T]_s$. Since  $P(y'+s^{-1}d^eG(y')T)\in P(y')+d^e(T)$ we get $P(y'+s^{-1}d^eG(y')T)=ds''$ for some $s''\in 1+(d,T)D[Y,T]$. It follows that $s''I\subset ((h,g):d)D[Y,T]_{ss'}$. On the other hand,
  $I\equiv I(y')\ \mbox{modulo}\ (d^e,h)D[Y,T]$ and $I(y')\subset d^{2e+1}D$. Thus $s''I\subset (0:_Vd)\cap d^eV=0$ because $(0:_Ad)\cap d^eA=0$ and the maps $A\to D$, $D\to V$ are flat, which shows our claim. It follows that
   $I\subset (h,g)D[Y,T]_{ss's''}$. Thus $E_{ss's''}\cong V_{s''} $ is a $B$-algebra which is also standard smooth over $D$ and $A$.

 As $\omega(s)\equiv 1$ modulo $d$ and $\psi(s'),\psi(s'')\equiv 1$ modulo $(d,t)$, $d,t\in \mm A'$ we see that $\omega(s),\psi(s'), \psi(s'')$ are invertible because  $A'$ is local  and $\psi$ (thus $v$) factors through the standard smooth $A$-algebra $E_{ss's''}$.

\begin{rem}  When $A'$ is the completion of $A$ then the algorithmic proof is much easier (one reason is given by Remark \ref{rem}) and it is somehow similar to the proof of Theorem \ref{gr}. Certainly, in this case the next algorithm could be substantially easier.
\end{rem}

\begin{rem}
If we want to study the case when $\dim A=2$ then we need first to treat the case when $\dim A=0$, $\dim A'=1$ and after that
the case when $A,\ A'$ are one dimensional but not local rings.
We expect that an algorithmic proof in the last case is  a hard goal.
 However, if we are lucky to get such proof it is doubtful that the corresponding algorithm will really work.
 \end{rem}
\vskip 0.5 cm
\section{The algorithm}
We obtain the following algorithm:

\begin{algorithm}[H]\label{alg:NeronDesing}
\begin{algorithmic}[1]
\REQUIRE $N\in \ZZ_{>0}$ a bound\\
$A:=k[x]_{( x)}/J, J=( h_1, \ldots, h_g) \subseteq k[x],  x=(x_1, \ldots, x_t), k$ an infinite field\\
$k':=Q(k[U]/\overline{J}), \overline{J}=( a_1, \ldots, a_r) \subseteq k[U], U=(U_1, \ldots, U_w)$ separable over $k$\\
$B:=A[Y]/I, I=( g_1, \ldots, g_l)\subseteq k[x, Y], Y=(Y_1, \ldots, Y_n)$\\
$v:B\to A'\subseteq K[[x]]/JK[[x]]$ an $A$--morphism, given by $ y'_1, \ldots, y'_n\in k[U, x]$, approximations $\mod( x)^N$ of $v(Y_i), K\supseteq k'$ a field.
\ENSURE A Neron desingularization of $v:B\to A'$ or the message ''the algorithm fails since the bound $N$ is too small''
\STATE Compute $P_1,\ldots,P_s$ the minimal associated primes of $A$.
\STATE Compute $w=(a_{i_1}, \ldots, a_{i_p})$ and a $p$--minor $\rho$ of $\left(\frac{\partial a_{i_v}}{\partial U_j}\right)$ such that $\rho\not\in \overline{J}$.\\
Compute $\tau\in( w):\overline{J}$ such that $k[U]_{\rho\tau}/( a_1, \ldots, a_r)=k[U]_{\rho\tau}/( w)$.\\
$D:=A[U]_{\rho\tau}/( w)$.
\STATE Compute $H_{B/A}$ and $H_{B/A}\cap A$.
\STATE If  dim$(A/H_{B/A}\cap A)=1$\\
$B:=S_B(I/I^2)$, $v$ trivially extended, write $B=A[Y]/I$, $n:=\#Y$\\
$Y:=Y,Z$, $I:=(I,Z)$, $B:=A[Y]/I,\  Z=(Z_1\ldots Z_n)$ , $v$ trivially extended
\STATE Compute $f=(f_1, \ldots, f_r)$ in $I$ such that $v((( f): I)\Delta_f)\not\subseteq P_i+( x)^N$ for all $i$
\STATE Complete $\left(\frac{\partial f_i}{\partial Y_j}\right)$ by random vectors from $k^n$ to obtain a square matrix $H$ with $v(\det(H))\not\in P_i+( x) ^N$ for all $i$
\STATE Compute $R\in (f):I$ such that $v(R)\not\in P_i+( x)^N$ for all $i$.
\STATE $P:=R\cdot \det(H)$ , compute $(P)\cap A$

%\STATE Compute $\widetilde{d}\in \langle x\rangle\ $ and  $l\in \ZZ_{>0}$,$\ \widetilde{d}\not\in P_i$ for $P_i$ minimal, such that $\langle x\rangle^l\subseteq \langle\widetilde{d}\rangle+J$ \\
%compute $c\in \ZZ_{\geq 0}$ such that $\widetilde{d}^c\in \langle v(P')\rangle+\langle x\rangle ^{2lc}$\\
%$d':=\widetilde{d}^c$
\STATE If dim$(A/(P)\cap A)=1$\\
Compute $d$ in $v(P)\cap A$ an active element.
 $f_{r+1}:=PY_{n+1}-d$,\\ $ B:=B[Y_{n+1}]/f_{r+1}$
$ f:=f, f_{r+1}\ ,\ Y:=Y,Y_{n+1}, y':= y',\frac{d}{v(P)}\ ,\ n:= n+1$,\ r:=r+1 \\
$H:=\begin{pmatrix}H & 0\\ \ast & P\end{pmatrix}$
$d:=d^2$,$R:=RY^2_n$, $P:=R\cdot \det(H)$
\STATE Else\\
Compute $d$ in $((P)\cap A)$ an active element, $P:=d$.
\STATE Compute $e$ such that $(0:d^e)=(0:d^{e+1})$
\STATE If $( x)^N\not\subseteq ( d^{2e+1})$ return ''the algorithm fails since the bound is too small''
\STATE Compute $b:=\frac{f(y')}{d^{e+1}}$ in $d^eD^r$
\STATE Compute $G'$ the adjoint matrix of $H\ ,\ G:=R G'$
\STATE Compute $s\in D$ such that $P(y')=ds$\\
$h:=s(Y-y')-d^eG(y')T\ ,\ T=(T_1, \ldots, T_n)$
\STATE $p:=\max\{\deg(f_i)\}_{i=1, \ldots, r}$\\
write $s^p f(Y)-s^p f(y')=\sum\limits_j s^{p-1} d^e\frac{\partial f}{\partial Y_j} (y') G_j(y') T_j +d^{2e} Q$\\
define $g_i:=s^p b_i+s^pT_j+d^{e-1}Q_i$ and $E:=D[Y, T]/( I, g, h)$
\STATE Compute $s'$\ the $r$--minor of $\left(\frac{\partial g}{\partial T}\right)$ given by the first $r$ variables of $T$
\STATE find $s''\in 1 +( d, T)_{D[Y,T]}$ such that $I\subset ( h,g) D[Y, T]_{ss' s''}$\\
($s''$ is given by $P(y'+s^{-1} d^eG(y')T)=d s'')$
\STATE return $E_{s s' s''}$
%\RETURN ($$)
\end{algorithmic}
\caption{Neron Desingularization}
\end{algorithm}

\begin{exmp} We assume

 $N=12$,

$A={\bf Q}[x_1,x_2]_{(x_1,x_2)}/(x_1x_2)$, $J=(x_1,x_2)$,

$k'={\bf Q}$,

$B=A[Y_1,Y_2]/(x_2Y_1-x_1Y_2)$, $I=(x_2Y_1,x_1Y_2)$,

$v:B\to {\bf Q}[[x_1,x_2]]/(x_1x_2)$, $v(Y_1)=x_1u$, $v(Y_2)=x_2w$,

$u,w\in {\bf Q}[[x_1,x_2]]$ algebraically independent,

$y'_1=\mbox{jet}(x_1u,11)$, $y'_2=\mbox{jet}(x_2w,11)$.

Now we follow the steps of the algorithm.

1. $P_1=(x_1)$, $P_2=(x_2)$,

2. $D=A$,

3. $H_{B/A}=(x_1,x_2)B$, $H_{B/A}\cap A=(x_1,x_2)$,

4. is not true,

5. $f=x_2Y_1+x_1Y_2$,

6. $H:=\begin{pmatrix}x_2 & x_1\\ -1 & 1\end{pmatrix}$

7. $R=x_1+x_2$,

8. $P=(x_1+x_2)^2$, $(P)\cap A=(x_1+x_2)^2$,

9. is not true,

10. $P=d=(x_1+x_2)^2$,

11. $e=1$,

12. is not true $(x_1,x_2)^{12}\subset (x_1x_2,(x_1,x_2)^6)$,

13. $f(y')=0$, $b=0$,

14. $G':=\begin{pmatrix}1 & -x_1\\ 1 & x_2\end{pmatrix}$, $G=(x_1+x_2)\begin{pmatrix}1 & -x_1\\ 1 & x_2\end{pmatrix}$,

15. $s=1$, $h=Y-y'-d(x_1+x_2)\begin{pmatrix}1 & -x_1\\ 1 & x_2\end{pmatrix}\begin{pmatrix} T_1\\ T_2\end{pmatrix}$,

16. $p=1$, $x_2(Y_1-y'_1)+x_1(Y_2-y'_2)=d^2T_1\ \mbox{mod} h $,

$g=T_1$, $E=A[Y,T]/(I,g,h)=A[Y,T]/(g,h)$,

17. $s'=1$,

18. $s''=1$,

19. return $A[Y_1,Y_2,T_2]/(Y_1-y'_1-x_1^4T_2, Y_2-y'_2-x_2^4T_2)$.

\end{exmp}

Thus the General Neron Desingularization $B'$ of $B$ could be $A[T_2]$. Let $\tau:B\to B'$ be the map given $Y_i\to y'_i+x_i^4T_2$, $i=1,2$. We may find $\rho_i\in (x_i)Q[[x_i]]$ such that $x_1u-y'_1=x_1^{11}\rho_1$ and $x_2w-y'_2=x_2^{11}\rho_2$. Set $t_2=d^4(\rho_1+\rho_2)$ and let $\tau':B'\to A'$ be given by $T_2\to t_2$. Clearly, $v=\tau'\tau$, that is $v$ factors through $B'$.

\section{An extension of  Greenberg's theorem on strong approximation}

Let $(A,\mm)$ be a Noetherian local ring  of dimension one, $A'={\hat A}$ the completion of $A$, $B=A[Y]/I$, $Y=(Y_1,\ldots,Y_n)$ an  $A$-algebra of finite type and $c \in \bf N$. If $A$ is Henselian excellent DVR then Greenberg \cite{Gr} showed that there exists a linear map $\nu:{\bf N}\to {\bf N}$ such that for each $A$-morphism $v:B\to A/\mm^{\nu(c)}$ there exists  an $A$-morphism $v':B\to A$ such that $v'\equiv v$ modulo $\mm^c$. More general, if $A$ is a DVR then
there exists a linear map $\nu:{\bf N}\to {\bf N}$ such that for each $A$-morphism $v:B\to A/\mm^{\nu(c)}$ there exists  an $A$-morphism $v':B\to A'$ such that $v'\equiv v$ modulo $\mm^cA'$.

 The aim of this section is to give a result  of Greenberg's type for one dimensional rings when the Jacobian locus is not {\em too small}. Let $e$ be as in the beginning of Section 3 and $\rho\in {\bf N}$. We  will show that the map $\nu$ given by $c\to (e+1)(\rho+1)+c$ works in the special case below.
Suppose that there exists  an $A$-morphism $v:B\to A/\mm^{(e+1)(\rho+1)+c}$
such that $(v( ((f):I)\Delta _f))\supset \mm^{\rho}/\mm^{(e+1)(\rho+1)+c}$  for some $f=(f_1,\ldots ,f_r)$, $r\leq n$ in $I$.

\begin{thm} \label{gr} Then there exists an $A$-morphism $v':B\to {\hat A}$ such that $v'\equiv v\ \mbox{modulo}\ \mm^c$, that is $v'(Y+I)\equiv v(Y+I)\ \mbox{modulo}\ \mm^c$.
\end{thm}
\begin{proof} We note that the proof of Theorem \ref{m} can work somehow in this case. Let $y'\in A^n$ be an element inducing $v(Y+I)$. Then $\mm^{\rho}\subset (((f):I)\Delta _f)(y'))+\mm^{(e+1)(\rho+1)+c}\subset ((((f):I)\Delta _f)(y'))+\mm^{(e+1)(\rho+2)+2c}\subset \ldots$ by hypothesis. It follows that  $\mm^{\rho}\subset ((((f):I)\Delta _f)(y'))$.

 As in Step 4 of Section 1 we may complete the Jacobian matrix $(\partial f/\partial Y)$  with $n-r$ rows from $k^n$ obtaining a square $n$ matrix $H$ with $v(\det H) \not \in \cup_{p\in \Min A'} p$

Set $d=(\det H)(y')$.
Next we follow the proof of Theorem \ref{m} with $D=A$, $s=1$, $P=L\det H$ and  $G$ such that
$$GH=HG=P\mbox{Id}_n$$
and so
$$d\mbox{Id}_n=P(y')\mbox{Id}_n=G(y')H(y').$$

Let  $$h=Y-y'-d^eG(y')T,$$
 where  $T=(T_1,\ldots,T_n)$ are new variables. We have
  $$f(Y)-f(y')\equiv
d^eP(y')T+d^{2e}Q$$
modulo $h$ where $Q\in T^2 A[T]^r$. But $f(y')\in \mm^{(e+1)(\rho+1)+c}A^r\subset d^2\mm^{c+e-1}A^r$ and we get $f(y')=d^2b$ for some $b\in \mm^cA^r$. Set $g_i=b_i+T_i+d^{e-1}Q_i$, $i\in [r]$ and $E=A[Y,T]/(I,h,g)$. We have an $A$-morphism $\beta:E\to A/\mm^c$ given by $(Y,T)\to (y',0)$  because $I(y')\equiv 0\ \mbox{modulo}\ \mm^c$, $h(y',0)=0$ and $g(0)=b\in \mm^cA^r$.

As in the proof of Theorem \ref{m} we see that
$I\subset (h,g)A[Y,T]_{s's''}$ for some  $s''\in 1+(d,T)A[Y,T]$. Indeed, we have $PI\subset (h,g)A[Y,T]$ and so $P(y'+d^eG(y')T)I\subset (h,g)A[Y,T]$. Since  $P(y'+d^eG(y')T)\in P(y')+d^e(T)$ we get $P(y'+d^eG(y')T)=ds''$ for some $s''\in 1+(d,T)A[Y,T]$. It follows that $s''I\subset ((h,g):d)A[Y,T]_{s'}$. On the other hand,
  $I\equiv I(y')\ \mbox{modulo}\ (d^e,h)A[Y,T]$ and $I(y')\subset \mm^{e\rho}\subset d^eA$. Set $V=(A[Y,T]/(h,g))_{s'}$.
  Then $s''I\subset (0:_Vd)\cap d^eV=0$ because $(0:_Ad)\cap d^eA=0$ and the map $A\to  V$ is flat, which shows our claim.

Then  $E_{s's''}\cong V_{s''} $ is a $B$-algebra which is also standard smooth over $A$ and $\beta$ extends to a map $\beta':E_{s's''}\to A/\mm^c$ as in the proof of Theorem \ref{m}. By the Implicit Function Theorem    $\beta'$ can be lifted to a map $w:E_{s's''}\to \hat A$ which coincides with $\beta'$ modulo $\mm^c$. It follows that the composite map $v'$, $B\to  E_{s's''}\xrightarrow{w} \hat A$ works.
\hfill\ \end{proof}

\begin{cor} In the assumptions of the above theorem, suppose that $(A,\mm)$ is excellent Henselian. Then there exists an $A$-morphism $v'':B\to A$ such that $v''\equiv v\ \mbox{modulo}\ \mm^c$, that is $v''(Y+I)\equiv v(Y+I)\ \mbox{modulo}\ \mm^c$.
\end{cor}
\begin{proof}
An excellent Henselian local ring $(A,\mm)$ has the Artin approximation property by \cite{P}, that is the solutions in $A$ of a system of polynomial equations $f$ over $A$ are dense in the set of the solutions of $f$ in $\hat A$. By Theorem \ref{gr} we get an $A$-morphism $v':B\to \hat A$ such that $v'\equiv v \ \mbox{modulo} \ \mm^c$. Then there exists an $A$-morphism $v'':B\to A$ such that $v''\equiv v'\equiv v\ \mbox{modulo}\ \mm^c$ by the  Artin approximation property.
\hfill\ \end{proof}

\begin{rem} If $\dim A>2$ then \cite[Remark 4.7]{R} gives an example when there exist no linear map $\nu$ as above. If
$A=k[[x_1,x_2]]$ then Theorem \ref{gr} and the above corollary hold as \cite[Theorem 4.3]{R} shows.
\end{rem}

\begin{thm} \label{gr1} Let $(A,\mm)$ be a Noetherian local ring  of dimension one,  $e\in \bf N$ as in the beginning of Section 2,  $B=A[Y]/I$, $Y=(Y_1,\ldots,Y_n)$ an  $A$-algebra of finite type, $\rho\in \bf N$ and  $f=(f_1,\ldots,f_r)$ a system of polynomials from $I$. Suppose  that $A$ is excellent Henselian and there exists $y'\in A^n$ such that
  $I(y')\equiv 0\ \mbox{modulo}\ \mm^{\rho}$  and  $((((f):I)\Delta _f))(y')\supset \mm^{\rho}$. Then the following statements are equivalent:
\begin{enumerate}
\item{} there exists $y''\in A^n$ such that $I(y'')\equiv 0\ \mbox{modulo}\ \mm^{(e+2)(\rho+1)}$ and $y''\equiv y'\ \mbox{modulo}\ \mm^{\rho}$,
\item{}  there exists $y\in A^n$ such that $I(y)=0$ and $y\equiv y'\ \mbox{modulo}\ \mm^{\rho}$.
\end{enumerate}
\end{thm}
For the proof apply the above corollary and Theorem \ref{gr}.

\end{document}